\documentclass[11pt,a4]{amsart}
\usepackage{amssymb,amsfonts,amsmath,amsthm,mathabx}
\usepackage{mathrsfs}
\usepackage{mathtools}
\usepackage{color}
\usepackage{ esint }
\usepackage{float} 
\usepackage{comment}
\usepackage{biblatex}
\bibliography{reference}

\usepackage[a4paper]{geometry}

\usepackage{setspace}

\usepackage{hyperref}
\hypersetup{colorlinks=true,allcolors=black,pdfstartview=Fit,breaklinks=true}

\theoremstyle{definition}
\newtheorem{defn}{Definition}[section]
\newtheorem{rmk}{Remark}

\theoremstyle{plain}

\newtheorem{prop}[defn]{Proposition}
\newtheorem{thm}[defn]{Theorem}

\newtheoremstyle{exp}
{\topsep}
{\topsep}
{\normalfont}
{0pt}
{}
{.}
{ }
{\thmname{#1}\thmnumber{ #2}\textnormal{\thmnote{ (#3)}}}

\theoremstyle{exp}

\newtheoremstyle{named}{}{}{\itshape}{}{\bfseries}{.}{.5em}{\thmnote{#3}#1}
\theoremstyle{named}
\newtheorem*{namedtheorem}{}

\newcommand{\C}{\mathbb{C}}
\newcommand{\R}{\mathbb{R}}

\newcommand{\N}{\mathbb{N}}

\newcommand{\supp}{\mathrm{supp}}
\DeclareMathOperator{\sign}{sign}
\numberwithin{equation}{section}

\usepackage{enumerate}

\usepackage{bm}

\usepackage{graphicx}

\usepackage{microtype}
\begin{document}
	\title{fractional integral on Hardy spaces on product domains}
	\author{Yiyu Tang}
	\address{Faculty of Mathematics and Computer Science, Nicolaus Copernicus University, Chopin street 12/18, 87-100 Toruń, Poland}
	\email{ytang@mat.umk.pl}
	
	\subjclass[]{}
	
	\maketitle
	\begin{abstract}
		Using the vector-valued theory of singular integrals, we prove a Hardy--Littlewood--Sobolev inequality on product Hardy spaces $H^p_{\rm{prod}}$, which is a parallel result of the classical Hardy--Littlewood--Sobolev inequality. The same technique shows the $H^p_{\rm{prod}}$-boundedness of the iterated Hilbert transform. As a byproduct, shorter proofs of several recently discovered Hardy type inequalities on product Hardy spaces are obtained.
	\end{abstract}	

\section{Introduction}
\subsection{The Hardy--Littlewood--Sobolev inequality}
	Let $0<\alpha<d$, define the \textit{fractional integral} of a function $f$ by
	\begin{equation*}
		I_{\alpha}f(x)\coloneqq
		\int_{\R^d}\frac{f(x-y)}{|y|^{d-\alpha}}\,\mathrm{d}y.
	\end{equation*}
	The classical Hardy--Littlewood--Sobolev inequality (\cite{Stein_Singular_integrals}, Chapter~\RN{5}, Theorem~1) says that
	\begin{equation*}
		\|I_{\alpha}f\|_{L^q}\lesssim\|f\|_{L^p},\quad\frac{1}{q}=\frac{1}{p}-\frac{\alpha}{d},
	\end{equation*}
	as long as $1<p<q<\infty$.

 It is known that many results of $L^p$ spaces (for example, the singular integral theory) have counterpart results in Hardy spaces $H^p$, and the fractional integral is not an exception.
There are several equivalent definitions of $H^p$, maybe the most straightforward one is from the Poisson maximal function (undefined notations can be found in Section~\ref{Section: Notations})
 \begin{equation}\label{Definition of ordinary Hardy spaces}
 	\begin{aligned}
 		H^p(\R^d)
 		\coloneqq
 		\Big\{f\in\mathcal{S}^\prime(\R^d):\sup_{\delta>0}\big|f\ast P_\delta (x)\big|\in L^p(\R^d)\Big\}.
 	\end{aligned}
 \end{equation}

For $0<p\leq1$, if we assume that $f$ has higher moment cancellation, say $\int_{\R^d}x^\gamma f(x)\,\mathrm{d}x=0$ for all $|\gamma|\leq d(p^{-1}-1)$, then $I_\alpha(f)$ can be analytically
 continued to $0<\alpha\leq d/p$. An $H^p\mapsto H^q$ version of the Hardy--Littlewood--Sobolev inequality says that
	\begin{equation}\label{Hardy--Littlewood--Sobolev inequality in ordinary H^p spaces}
		\|I_{\alpha}f\|_{H^q}\lesssim\|f\|_{H^p},\quad\frac{1}{q}=\frac{1}{p}-\frac{\alpha}{d},
	\end{equation}
	as long as $0<p<q<\infty$. This is proved by Krantz in 1982 (\cite{Krantz}, Corollary~2.3) using atomic decomposition of $H^p$.
\subsection{Hardy space on product domain}
 In this paper, our main focus is on the product Hardy spaces $H^p_{\mathrm{prod}}(\R^d)$, which are defined by
	\begin{equation}\label{Definition of product Hardy spaces}
		\begin{aligned}
			H^p_{\mathrm{prod}}(\R^d)
			\coloneqq
			\Big\{f\in\mathcal{S}^\prime(\R^d):\sup_{\delta_1,\ldots,\delta_d>0}\big|f\ast (\otimes_{i=1}^dP_{\delta_i}) (x)\big|\in L^p(\R^d)\Big\}.
		\end{aligned}
	\end{equation}
We want to know if there exists a Hardy--Littlewood--Sobolev inequalty on $H^p_{\mathrm{prod}}(\R^d)$. In the definition \eqref{Definition of ordinary Hardy spaces}, there is only one parameter $\delta>0$, so the original operator $I_\alpha$ requires certain modification to coherent with the multi-parameter nature of $H^p_{\mathrm{prod}}(\R^d)$. A nature candidate is the following operator:
	\begin{equation}\label{Definition of the product form of the fractional integral operator}
	I_{(\alpha,d)}f(x)\coloneqq
	\int_{\R^d}\frac{f(x-y)}{|y_1y_2\cdots y_d|^{1-\frac{\alpha}{d}}}\,\mathrm{d}y.
\end{equation}
We call it the \textit{product form of the fractional integral}. This choice is not artificial, when $1<p,q<\infty$, weighted $L^p\mapsto L^q$ inequalities of $I_{(\alpha,d)}$ have been studied systematically by Sawyer and Wang\footnote{Strictly speaking, they studied the fractional integral on product domain $\R^m\times\R^n$, which corresponds to $I_{\alpha,2}$ in spirit.} in \cite{Sawyer_Wang} and \cite{Sawyer_Wang_2}. Now one may want to proceed Krantz's proof to obtain  $H^p_{\rm{prod}}\mapsto H^q_{\rm{prod}}$ inequalities of $I_{(\alpha,d)}$. However, there are some difficulties, which we shall now briefly describe.

\subsubsection*{Differences between $H^p$ and $H^p_{\rm{prod}}$}
One way to describe the differences between $H^p$ and $H^p_{\rm{prod}}$ is from the atomic decomposition. An $H^p(\R^d)$-atom is a function $a$ satisfies:
\begin{itemize}
	\item There exists a ball $B\in\R^d$, so that $\supp(a)\subset B$, and $\|a\|_\infty\leq|B|^{-1/p}$.
	\item The integral $\int x^\gamma a(x)=0$, for all multi-indices $\gamma$ with $|\gamma|\leq [\frac{d}{p}-d]$.
\end{itemize}
The \textit{atomic decomposition} of $H^p(\R^d)$ states that
\begin{equation*}
	\|f\|_{H^p(\R^d)}
	\approx
	\inf\Big\{
	\Big(\sum_{j=1}^\infty|\lambda_j|^p\Big)^{1/p}:
	\lim_{N\rightarrow\infty}
	\|\sum_{j=1}^N\lambda_ja_j-f\|_{H^p}
	=0,
	\text{ where }
	a_j
	\text{ are }
	H^p
	\text{-atoms.}
	\Big\}.
\end{equation*}
The $H^p_{\rm{prod}}$-atoms are far more complicated. We only describe the case $d=2$, which is due Chang and Fefferman (see \cite{Chang_Feerman_Annals_Paper}, and the survey \cite{Chang_Feerman_BAMS_survey}). An $H^p_{\rm{prod}}(\R^2)$-atom is a function $a$ satisfies
\begin{itemize}
	\item  The function $a$ is supported in an open set $\Omega$ of finite measure, and $\|a\|_2^2
	\leq
	|\Omega|^{1-\frac{2}{p}}$.
	\item The function $a$ can be further decomposed as $\sum_{R\in m(\Omega)}a_R$, where $m(\Omega)$ are maximal dyadic rectangles of $\Omega$, so that
	\begin{itemize}
		\item The support of  $a_R$ is a subset of $3R$, and $\sum_{R\in m(\Omega)}\|a_R\|_2^2
		\leq
		|\Omega|^{1-\frac{2}{p}}$.
		\item The integrals $\int_{\R}a_R(x,y)x^\alpha\,\mathrm{d}x=0$, for all $\alpha\leq[\frac{1}{p}-1]$ and all $y\in\R$.
		\item The integrals $\int_{\R}a_R(x,y)y^\beta\,\mathrm{d}y=0$, for all $\beta\leq[\frac{1}{p}-1]$ and all $x\in\R$.

	\end{itemize}
\end{itemize}
The \textit{atomic decomposition} of $H^p_{\rm{prod}}(\R^d)$ states that
\begin{equation*}
	\|f\|_{H^p_{\rm{prod}}}
	\approx
	\inf\Big\{
	\Big(\sum_{j=1}^\infty|\lambda_j|^p\Big)^{1/p}:
	\lim_{N\rightarrow\infty}
	\|\sum_{j=1}^N\lambda_ja_j-f\|_{H^p_{\rm{prod}}}
	=0,
	\text{ where }
	a_j
	\text{ are }
	H^p_{\rm{prod}}
	\text{-atoms.}
	\Big\}.
\end{equation*}

In Krantz's proof of \eqref{Hardy--Littlewood--Sobolev inequality in ordinary H^p spaces}, the following geometric property is used:\textit{ Let $B$ be a Euclidean ball centered at $0$, if $x\in B$ and $y\notin 2B$, then $|x-y|\approx |y|$}. The geometry of a general open set is far more complicated than balls, so no such property is available, and a step-by-step repetition of Krantz's proof does not work.

This is not the only difficulty. Krantz's proof is essentially the same as that of the boundedness of singular integral operator (i.e., the Calder\'on--Zygmund theory) on $H^p$ spaces, which has been well established. However, when move to the product settings, an example of Journ\'e (see \cite{Journe_Multiparameter_case}) reveals that the Calder\'on--Zygmund theory on $H^p_{\rm{prod}}(\R^2)$ and $H^p_{\rm{prod}}(\R^3)$ are different.

Nevertheless, there are still ways to circumvent these difficulties. In this paper, we show that, if we examine $H^p_{\rm{prod}}$ from the vector-valued point of view, then the $H^p_{\rm{prod}}\mapsto H^q_{\rm{prod}}$ boundedness can be easily established. Moreover, we do not need to resort deep results of $H^p_{\rm{prod}}$.
\begin{thm}\label{Hardy--Littlewood--Sobolev inequality in product H^p spaces}
	Let $I_{(\alpha,d)}$ be the product form of the fractional integral operator.  We have the Hardy--Littlewood--Sobolev inequality in product form:
	\begin{equation*}
		\|I_{(\alpha,d)}f\|_{H_{\mathrm{prod}}^q}\lesssim\|f\|_{H_{\mathrm{prod}}^p},\quad\frac{1}{q}=\frac{1}{p}-\frac{\alpha}{d},
	\end{equation*}
	as long as $0<p<q<\infty$.
\end{thm}
Also, based on the same technique, we obtain shorter (both conceptually and technically)  proofs of the main results of a recent paper \cite{Dyachenko--Nursultanov--Tikhonov--Weisz} by Dyachenko, Nursultanov, Tikhonov and Weisz. An example is the $H^p_{\rm{prod}}$ Hardy--Littlewood inequality 
\begin{equation}\label{Product Hardy--Littlewood inequality}
		\bigg(\int_{\R^d}\frac{|\widehat{f}(\xi)|^p}{|\xi_1\xi_2\cdots\xi_d|^{2-p}}\,\mathrm{d}\xi\bigg)^{\frac{1}{p}}
		\lesssim_{d,p}
		\|f\|_{H^p_{\mathrm{prod}}(\R^d)},\text{ where }0<p\leq 1.
\end{equation}
This is a parallel result of the classical $H^p$ Hardy--Littlewood inequality
\begin{equation}\label{Classical Hardy--Littlewood inequality}
	\bigg(\int_{\R^d}\frac{|\widehat{f}(\xi)|^p}{|\xi|^{(2-p)d}}\,\mathrm{d}\xi\bigg)^{\frac{1}{p}}
	\lesssim_{d,p}
	\|f\|_{H^p(\R^d)},\text{ where }0<p\leq 1.
\end{equation}
The inequality \eqref{Classical Hardy--Littlewood inequality} is well-known, and can be found in \cite{Stein_Harmonic_Analysis}, Chapter \RN{3}, Section 5.4, Statement (d). The inequality \eqref{Product Hardy--Littlewood inequality} is less known. Using $K$-functional and interpolation, Jawerth and Torchinsky gave a sketchy proof of \eqref{Product Hardy--Littlewood inequality} for $d=2$ (see \cite{Jawerth--Torchinsky}, Proposition C). However, according to \cite{Dyachenko--Nursultanov--Tikhonov--Weisz}, Jawerth--Torchinsky's proof has some gaps, and the general case $d\geq2$ of \eqref{Product Hardy--Littlewood inequality} is proved in \cite{Dyachenko--Nursultanov--Tikhonov--Weisz},  Theorem 4, by using singular integral theory on $H^p_{\rm{prod}}$, which requires a considerable preparatory work. We will show that the inequality \eqref{Product Hardy--Littlewood inequality} can be proved only by iterated use of the Fubini's theorem and the 1-dimensional theory, hece the application of the Calder\'on--Zygmund theory on $H^p_{\rm{prod}}$ can be avoided.

\section{Notations}\label{Section: Notations}
Some standard notations in analysis (like $L^p$ spaces, Schwartz class $\mathcal{S}$, and Hilbert transform $H$, etc) will not be repeated here. Below $X$ is a Banach space.
	\begin{itemize}
	\item $H^p(\R^d;X)$: the vector-valued Hardy spaces. Several equivalent definitions will be described subsequently. We also use $H^p,\,H^p(\R^d)$ when $X=\R$ or $\C$.
	
	\item $H^p_{\mathrm{prod}}(\R^d;X)$: the vector-valued product Hardy spaces. We use the conventions $H^p_{\mathrm{prod}},\,H^p_{\mathrm{prod}}(\R^d)$ when $X=\R$ or $\C$.
	
	\item $\mathcal{F}_{i}(f)$: the partial Fourier transform of a function $f\colon \R^d \rightarrow \C$,
	\begin{equation*}
		\mathcal{F}_{i}(f)(\xi_i)
		\coloneqq
		\int_{\mathbb{R}}f(x)\mathrm{e}^{-2\pi\mathrm{i}x_i\cdot\xi_i}\mathrm{d}x_i.
	\end{equation*}
	\item $H_{i}(f)$: the directional Hilbert transform of a function 
	$f \colon \R^d \rightarrow \C$, defined by
	\begin{equation*}
		H_i(f)(x)\coloneqq	\frac{1}{\pi}\mathrm{ p.v.}\int_{\R}\frac{f(x_1,\ldots,y_i,\ldots,x_d)}{x_i-y_i}\,\mathrm{d}y_i.
	\end{equation*}
Equivalently, we have $H_i(f)=(-\mathrm{i}\sign(\xi_i)\widehat{f}(\xi))^{\widecheck{\,}}$.

\item  $P_t:$ the Poisson kernel on $\R^d$, defined by $\widehat{P_t\,}(\xi)\coloneqq\mathrm{e}^{-t|\xi|}$, where $t>0$
\item  $P_{t_1,\ldots,t_d}$ or $\otimes_{i=1}^dP_{\delta_i}$: the product Poisson kernel on $\R^d$, defined by $\prod_{i=1}^dP_{t_i}(x_i)$.

\item  $\mathcal{H}$: the  Hardy--Ces\`aro operator. For appropriate $f \colon \R^d \rightarrow \C$,
	\begin{equation*}
		\mathcal{H}(f)(x)
		\coloneqq
		\frac{1}{x_1x_2\cdots
			x_d}\int_{0}^{x_1}\int_{0}^{x_2}\ldots\int_{0}^{x_d}f(t_1,t_2,\ldots,t_d)\,\mathrm{d}t.
	\end{equation*}

\item $I_{(\alpha,d)}$: the product fractional integral operator, defined in \eqref{Definition of the product form of the fractional integral operator}.

\item $S(f)$ and $S_X(f)$: the Lusin area integral (or $S$-function) of a scalar-valued (or $X$-valued) function $f$, defined in Section~\ref{Section: Fractional integral on product spaces}.
\end{itemize}

\section{Fractional integral on product spaces}\label{Section: Fractional integral on product spaces}
In this section, we prove Theorem~\ref{Hardy--Littlewood--Sobolev inequality in product H^p spaces}. We only consider the case $d=2$ for simplicity. Although theories of $H^p_{\rm{prod}}(\R^2)$ and $H^p_{\rm{prod}}(\R^3)$ are essentially different, for our approach used in this paper, the dimension does not cause problems. Also, the 1-dimensional case is known, as there is no product theory when $d=1$.
	The method of the proof goes back at least as far as  \cite{Fefferman--Stein} (see the lemma on Page 139 there). Write 
	\begin{equation*}
		I_{(\alpha,2)}f(x,y)=f\ast_1K_1\ast_2K_2,
	\end{equation*}
	where 
	\begin{equation*}
		K_1(x)=\frac{1}{|x|^{1-\frac{\alpha}{2}}},\ 
		K_2(y)=\frac{1}{|y|^{1-\frac{\alpha}{2}}},
		\ 
		\text{and }x,y\in \R.
	\end{equation*}
The notation $\ast_1$ indicates that the convolution is taken with respect to $x$-variable, and $\ast_2$ is taken with respect to $y$-variable.
	
	Choose an even function $\psi\in C^\infty_c(\R)$ satisfies the following conditions:
	\begin{equation}\label{Defining function of Lusin area integral}
		\begin{aligned}
			\mathrm{supp}(\psi)\subset\{|x|\leq1\},\textrm{ and }\int x^\alpha\psi(x)\,\mathrm{d}x=0\text{ for all }0\leq\alpha\leq N,
		\end{aligned}
	\end{equation}
	where\footnote{Later we will use the $S$-function to characterize both $H^p_{\rm{prod}}$ and $H^q$. Recall that $p<q$, so to ensure that $\psi$ has a sufficiently high-order cancellation adapted to $H^p_{\rm{prod}}$,  we use $[1/p-1]$, rather than $[1/q-1]$.} $N\coloneqq [1/p-1]$.
	
	Define $\psi_{\rho}\coloneqq\frac{1}{\rho}\psi(\frac{\cdot}{\rho})$. The Lusin area integral (or $S$-function) of $	I_{(\alpha,2)}f$ is
	\begin{equation*}
		\begin{aligned}
			S(I_{(\alpha,2)}f)(x,y)
			&\coloneqq
			\bigg(\int_{\Gamma_2}\int_{\Gamma_1}
			\big|I_{(\alpha,2)}f\ast_1\psi_{\rho_1}\ast_2\psi_{\rho_2}\big|^2(x-s,y-t)
			\,\frac{\mathrm{d}s\mathrm{d}\rho_1}{\rho_1^2}\,\frac{\mathrm{d}t\mathrm{d}\rho_2}{\rho_2^2}\bigg)^{\frac{1}{2}},\\
		\end{aligned}
	\end{equation*}
	where $\Gamma_i$ is the cone $\Gamma_i\coloneqq\{(s,\rho_i)\in\R^2:|s|<\rho_i\}$.
	
	The square function characterization of $H^q_{\mathrm{prod}}$ says that $\|I_{(\alpha,2)}f\|_{H^q_{\mathrm{prod}}}
	\approx
	\|S(I_{(\alpha,2)}f)\|_{L^q}$. By definition of $I_{(\alpha,2)}$, we can rewrite the $S(I_{(\alpha,2)}f)$ as
	\begin{equation*}
		\begin{aligned}
			\bigg(\int_{\Gamma_2}\bigg(\int_{\Gamma_1}
			\bigg|\psi_{\rho_2}\ast_2K_2\ast_2\Big[(f\ast_1K_1\ast_1\psi_{\rho_1})(x-s,\cdot)\Big]\bigg|^2(y-t)
			\,
			\frac{\mathrm{d}s\mathrm{d}\rho_1}{\rho_1^2}\bigg)\,\frac{\mathrm{d}t\mathrm{d}\rho_2}{\rho_2^2}\bigg)^{\frac{1}{2}}.\\
		\end{aligned}
	\end{equation*}
	We define $F_{u,\rho_1}(y)
	\coloneqq
	[(f\ast_1K_1\ast_1\psi_{\rho_1})(\cdot,y)](u)$,
	and $X\coloneqq L^2(\Gamma_1,\rho_1^{-2}\,\mathrm{d}s\mathrm{d}\rho_1)$. 
	
	Now the integral $\int_{\Gamma_1}|\psi_{\rho_2}\ast_2K_2\ast_2(f\ast_1K_1\ast_1\psi_{\rho_1})(x-s,\cdot)|^2(y-t)\rho_1^{-2}\,\mathrm{d}s\mathrm{d}\rho_1$ can be written as
	\begin{equation*}
		\begin{aligned}
			\Big\|\psi_{\rho_2}\ast_2K_2\ast_2F_{x-s,\rho_1}\Big\|^2_{X}(y-t),
		\end{aligned}
	\end{equation*}
	and
	\begin{equation*}
		\begin{aligned}
			S(I_{(\alpha,2)}f)(x,y)
			=
			\bigg(\int_{\Gamma_2}\Big\|\psi_{\rho_2}\ast_2K_2\ast_2F_{x-s,\rho_1}\Big\|^2_{X}(y-t)\,\frac{\mathrm{d}t\mathrm{d}\rho_2}{\rho_2^2}\bigg)^{\frac{1}{2}}.\\
		\end{aligned}
	\end{equation*}
	Our goal is to calculate $\iint|S(I_{(\alpha,2)}f)|^q\,\mathrm{d}y\mathrm{d}x$, let us first calculate the $\int\mathrm{d}y$ integral:
	\begin{equation*}
		\bigg(\int_{\R}|S(I_{(\alpha,2)}f)|^q\,\mathrm{d}y\bigg)^\frac{1}{q}
		=
		\bigg(\int_{\R}	\bigg(\int_{\Gamma_2}\Big\|\psi_{\rho_2}\ast_2K_2\ast_2F_{x-s,\rho_1}\Big\|^2_{X}(y-t)\,\frac{\mathrm{d}t\mathrm{d}\rho_2}{\rho_2^2}\bigg)^{\frac{q}{2}}\,\mathrm{d}y\bigg)^\frac{1}{q}.
	\end{equation*}
	We will use the $S$-function characterization of $X$-valued Hardy spaces $H^q(\R;X)$. Let $g\colon \R\rightarrow X$ be a vector-valued function, its $S$-function is defined by
	\begin{equation*}
		S_X(g)(y)
		\coloneqq
		\bigg(\int_{\Gamma_2}\Big\|\psi_{\rho_2}\ast g\Big\|^2_{X}(y-t)\,\frac{\mathrm{d}t\mathrm{d}\rho_2}{\rho_2^2}\bigg)^{\frac{1}{2}}.
	\end{equation*}
	An equivalent definition of $H^q(\R;X)$ is
	\begin{equation*}
		H^q(\R;X)
		\coloneqq
		\{f\in\mathcal{S}^\prime(\R;X):S_X(f)\in L^q(\R,\mathrm{d}y)\}.
	\end{equation*}
	We do not give the definition of $\mathcal{S}^\prime(\R;X)$ here, which can be found in \cite{Hytonen_book}, Chapter~2, Section~2.4.
	We have
	\begin{equation*}
		\bigg(\int_{\R}|S(I_{(\alpha,2)}f)|^q\,\mathrm{d}y\bigg)^\frac{1}{q}	
		=
		\|S_X(K_2\ast_2F_{x-s,\rho_1})\|_{L^q(\R,\mathrm{d}y)}
	\end{equation*}
	hence $(\int_{\R}|S(I_{(\alpha,2)}f)|^q\,\mathrm{d}y)^\frac{1}{q}
	\approx
	\|K_2\ast_2F_{x-s,\rho_1}\|_{H^q(\R,\mathrm{d}y;X)}$.
	
	The Hardy--Littlewood--Sobolev inequality for $X$-valued function also holds (one just replaces the Euclidean norm $|\cdot|$ by $\|\cdot\|_X$ and proceeds Krantz's proof in \cite{Krantz} step by step\footnote{For this, we need to introduce $H^p(\R;X)$-atoms. The reader can find it on page 338 of \cite{Blasco--Xu}. When $X$ is a Hilbert space, most properties of $H^p(\R^d)$ can be extended to $H^p(\R^d;X)$, like the atomic decomposition, smooth maximal function characterization, etc. We will use these as a black box.}). So we can use the Hardy--Littlewood--Sobolev inequality for $K_2$ on $H^q(\R,\mathrm{d}y;X)$ and get $\|K_2\ast_2F_{x-s,\rho_1}\|_{H^q(\R,\mathrm{d}y;X)}
		\lesssim
		\|F_{x-s,\rho_1}\|_{H^p(\R,\mathrm{d}y;X)}$.

Again, we have
	$\|F_{x-s,\rho_1}\|_{H^p(\R,\mathrm{d}y;X)}
	=
	\|S_X(F_{x-s,\rho_1})\|_{L^p(\R,\mathrm{d}y)}$ by $S$-function characterization. 	Therefore,
	\begin{equation*}
		\begin{aligned}
			\bigg(\iint_{\R\times\R}|S(I_{(\alpha,2)}f)|^q\,\mathrm{d}y\mathrm{d}x\bigg)^{\frac{1}{q}}
			&\lesssim
			\bigg(\int_{\R}\|F_{x-s,\rho_1}\|^q_{H^p(\R,\mathrm{d}y;X)}\,\mathrm{d}x\bigg)^{\frac{1}{q}}\\
			&\lesssim
			\bigg(\int_{\R}	\bigg(\int_\R |S_X(F_{x-s,\rho_1})|^p(y)\,\mathrm{d}y\bigg)^{\frac{q}{p}}\,\mathrm{d}x\bigg)^{\frac{1}{q}}\\
			&\lesssim
			\bigg(\int_{\R}	\bigg(\int_\R |S_X(F_{x-s,\rho_1})|^q(y)\,\mathrm{d}x\bigg)^{\frac{p}{q}}\,\mathrm{d}y\bigg)^{\frac{1}{p}}.\\
		\end{aligned}
	\end{equation*}
	We used the Minkowski's integral inequality (notice that $q>p$) in the last line.
	
	The $S$-function $S_X(F_{x-s,\rho_1})(y)$ is
	\begin{equation*}
		\begin{aligned}
			\bigg(\int_{\Gamma_2}\big\|\psi_{\rho_2}\ast_2(F_{x-s,\rho_1})
			\big\|_X(y-t)\,\frac{\mathrm{d}t\mathrm{d}\rho_2}{\rho_2^2}\bigg)^{1/2}
			=\bigg(\int_{\Gamma_1}
			\Big\|\psi_{\rho_2}\ast_1K_1\ast_1G_{y-t,\rho_2}\Big\|_Y^2(x-s)
			\,\frac{\mathrm{d}s\mathrm{d}\rho_1}{\rho_1^2}\bigg)^{\frac{1}{2}},\\
		\end{aligned}
	\end{equation*}
	where the function $G$, the space $Y$ are defined by
	\begin{equation*}	
		\begin{aligned}
			G_{v,\rho_2}(x)
			\coloneqq
			(f\ast_2\psi_{\rho_2})(x,v)
			=
			\Big[(f\ast_2\psi_{\rho_2})(x,\cdot)\Big](v),\quad
			Y
			\coloneqq
			L^2\Big(\Gamma_2,\frac{\mathrm{d}t\mathrm{d}\rho_2}{\rho_2^2}\Big).
		\end{aligned}
	\end{equation*}
	Therefore,  by $S$-function characterization of $H^q(\R,\mathrm{d}x;Y)$, and the vector-valued Hardy--Littlewood--Sobolev inequality for $K_1$,
	\begin{equation*}
		\begin{aligned}
			\bigg(\int_\R \bigg(\int_{\Gamma_1}
			\Big\|\psi_{\rho_2}\ast_1K_1\ast_1G_{y-t,\rho_2}\Big\|_Y^2(x-s)
			\,\frac{\mathrm{d}s\mathrm{d}\rho_1}{\rho_1^2}\bigg)^{\frac{q}{2}}\,\mathrm{d}x\bigg)^{\frac{1}{q}}
			\lesssim
			\|G_{y-t,\rho_2}\|_{H^p(\R,\mathrm{d}x;Y)}.
		\end{aligned}
	\end{equation*}
	We have proved that
	\begin{equation*}
		\begin{aligned}
			\bigg(\iint_{\R\times\R}|S(I_{(\alpha,2)}f)|^q\,\mathrm{d}y\mathrm{d}x\bigg)^{\frac{1}{q}}
			&\lesssim
			\bigg(\int_{\R}	\bigg(\int_\R |S_X(F_{x-s,\rho_1})|^q(y)\,\mathrm{d}x\bigg)^{\frac{p}{q}}\,\mathrm{d}y\bigg)^{\frac{1}{p}}\\
			&\lesssim
			\bigg(\int_{\R}	\| G_{y-t,\rho_2}\|^p_{H^p(\R,\mathrm{d}x;Y)}\,\mathrm{d}y\bigg)^{\frac{1}{p}}.
		\end{aligned}
	\end{equation*}
	The $S$-function of $ G_{y-t,\rho_2}$ is
	\begin{equation*}
		S_Y(G_{y-t,\rho_2})(x)
		\coloneqq
		\bigg(\int_{\Gamma_1}
		\Big\|\psi_{\rho_2}\ast_1G_{y-t,\rho_2}\Big\|_Y^2(x-s)
		\,\frac{\mathrm{d}s\mathrm{d}\rho_1}{\rho_1^2}\bigg)^{\frac{1}{2}},
	\end{equation*}
	which is exactly the $S$-function of $f$:
	\begin{equation*}
		S(f)(x,y)
		=
		\bigg(\iint_{\Gamma_1\times\Gamma_2}
		\big|f\ast_1\psi_{\rho_1}\ast_2\psi_{\rho_2}\big|^2(x-s,y-t)
		\,\frac{\mathrm{d}s\mathrm{d}\rho_1}{\rho_1^2}\,\frac{\mathrm{d}t\mathrm{d}\rho_2}{\rho_2^2}\bigg)^{\frac{1}{2}}.\\
	\end{equation*}
	To conclude,
	\begin{equation*}
		\begin{aligned}
			\|I_{(\alpha,2)}f\|_{H^q_{\mathrm{prod}}}
			\approx
			\|S(I_{(\alpha,2)}f)\|_{L^q}
			\lesssim
			\|S(f)\|_{L^p}
			\approx
			\|f\|_{H^p_{\mathrm{prod}}}.
		\end{aligned}
	\end{equation*}
Using the same method, one can prove the following proposition.
\begin{prop}[Boundedness of the iterated Hilbert transform]\label{H^p to L^p Boundedness of the iterated Hilbert transform}
	Assume that $d\geq1$ and $0<p\leq1$. The iterated Hilbert transform is bounded from $H^p_{\rm{prod}}$ to itself:
	\begin{equation*}
		\|H_1H_2\ldots H_df\|_{H^p_{\rm{prod}}(\R^d)}\lesssim\|f\|_{H^p_{\rm{prod}}(\R^d)}.
	\end{equation*}
\end{prop}
\begin{rmk}
	This result should not be new, but the author is not able to find a literature that states it precisely.
	The $H^p_{\rm{prod}}\mapsto L^p$ boundedness of iterated Hilbert transform is essentially due to Pipher. In Theorem~2.2 of \cite{Pipher}, it briefly mentioned that more smoothness of the kernel improves the range $p$ from $p_0<p\leq1$ to $p>0$. A precise statement can be found in \cite{Weisz}, Theorem~2. Both proofs rely heavily on Calder\'on--Zygmund theory on $H^p_{\rm{prod}}$.
\end{rmk}
\begin{proof}
	The 1-dimensional case is classical. We assume that $d=2$ for simplicity. Write  $H_1H_2f=f\ast_1k_1\ast_2k_2$,
	 where $k_1(x)=\frac{1}{x}$,\ 
	$k_2(y)=\frac{1}{y},$
	and $x,y\in \R.$ For $\psi$ as in \eqref{Defining function of Lusin area integral}, the $S$-function of 	$H_1H_2f$ is
	\begin{equation*}
		\begin{aligned}
			\bigg(\iint_{\Gamma_2\times \Gamma_1}
			\bigg|\psi_{\rho_2}\ast_2k_2\ast_2\Big[(f\ast_1k_1\ast_1\psi_{\rho_1})(x-s,\cdot)\Big]\bigg|^2(y-t)
			\,
			\frac{\mathrm{d}s\mathrm{d}\rho_1}{\rho_1^2}\frac{\mathrm{d}t\mathrm{d}\rho_2}{\rho_2^2}\bigg)^{\frac{1}{2}}.\\
		\end{aligned}
	\end{equation*}
	We define
	\begin{equation*}
		\begin{aligned}
			F_{u,\rho_1}(y)
			\coloneqq
			(f\ast_1k_1\ast_1\psi_{\rho_1})(u,y)
			=
			\Big[\big(f\ast_1k_1\ast_1\psi_{\rho_1}\big)(\cdot,y)\Big](u).
		\end{aligned}
	\end{equation*}
	Then (the space $X$ was defined in the proof of Theorem~\ref{Hardy--Littlewood--Sobolev inequality in product H^p spaces})
	\begin{equation*}
		\begin{aligned}
			S(H_1H_2f)(x,y)
			=
			\bigg(\int_{\Gamma_2}\Big\|\psi_{\rho_2}\ast_2k_2\ast_2F_{x-s,\rho_1}\Big\|^2_{X}(y-t)\,\frac{\mathrm{d}t\mathrm{d}\rho_2}{\rho_2^2}\bigg)^{\frac{1}{2}}.\\
		\end{aligned}
	\end{equation*}
	Our goal is to estimate $\iint|S(H_1H_2f)|^p\,\mathrm{d}y\mathrm{d}x$, let us first calculate the $\int\mathrm{d}y$ integral:
	\begin{equation*}
		\begin{aligned}
			\bigg(\int_{\R}|S(H_1H_2f)|^p\,\mathrm{d}y\bigg)^\frac{1}{p}
			=\|S_X(k_2\ast_2F_{x-s,\rho_1})\|_{L^q(\R,\mathrm{d}y)}
			\approx
			\|k_2\ast_2F_{x-s,\rho_1}\|_{H^p(\R,\mathrm{d}y;X)}.
		\end{aligned}
	\end{equation*}
	We will use the following proposition:
	\begin{prop}\label{Hilbert transform is bounded on X-valued Hardy spaces}
		Let $X$ be a Hilbert space, then the $X$-valued Hilbert transform $H$ is bounded from $H^p(\R;X)$ to itself for all $0<p\leq1$.
	\end{prop}
	\begin{proof}(Proof of Proposition~\ref{Hilbert transform is bounded on X-valued Hardy spaces})
This is identical\footnote{Of course, the vector-valued theory is very different from the scalar-valued theory, but we need not worry about these matters when $X$ is a Hilbert space (particularly UMD).} to the scalar-valued case, which can be found in \cite{MFA}, Chapter 2, Section 2.4, Theorem 2.4.5.
	\end{proof}

	Proposition~\ref{Hilbert transform is bounded on X-valued Hardy spaces} implies that $\|k_2\ast_2F_{x-s,\rho_1}\|_{H^p(\R,\mathrm{d}y;X)}
	\lesssim
	\|F_{x-s,\rho_1}\|_{H^p(\R,\mathrm{d}y;X)}$. So we have
	\begin{equation*}
		\begin{aligned}
			\bigg(\int_{\R}\int_{\R}|S(H_1H_2f)|^p\,\mathrm{d}y\mathrm{d}x\bigg)^\frac{1}{p}
			&\lesssim
			\bigg(\int_{\R}\|F_{x-s,\rho_1}\|_{H^p(\R,\mathrm{d}y;X)}^p\,\mathrm{d}x\bigg)^{\frac{1}{p}}\\
			&=
			\bigg(\int_{\R}	\int_\R |S_X(F_{x-s,\rho_1})|^p(y)\,\mathrm{d}x\mathrm{d}y\bigg)^{\frac{1}{p}}\\
			&=
			\bigg(\int_{\R}		\|k_1\ast_1 G_{y-t,\rho_2}\|^p_{H^p(\R,\mathrm{d}x;Y)}\,\mathrm{d}y\bigg)^{\frac{1}{p}},\\
		\end{aligned}
	\end{equation*}
	where
	$G_{v,\rho_2}(x)
	\coloneqq
	\big[(f\ast_2\psi_{\rho_2})(x,\cdot)\big](v)$.
	
	Again, Proposition~\ref{Hilbert transform is bounded on X-valued Hardy spaces} implies $\|k_1\ast_1 G_{y-t,\rho_2}\|_{H^p(\R,\mathrm{d}x;Y)}
	\lesssim
	\| G_{y-t,\rho_2}\|_{H^p(\R,\mathrm{d}x;Y)}$. To conclude,
	\begin{equation*}
		\begin{aligned}
			\|H_1H_2f\|_{H^p_{\mathrm{prod}}}
			\lesssim
			\bigg(\int_{\R}	\| G_{y-t,\rho_2}\|^p_{H^p(\R,\mathrm{d}x;Y)}\,\mathrm{d}y\bigg)^{\frac{1}{p}}
			\approx
			\|f\|_{H^p_{\mathrm{prod}}}.
		\end{aligned}
	\end{equation*}
\end{proof}

\section{New proof of several Hardy's inequalities on product spaces}\label{Section: Hardy's inequality on product spaces}
We collect main results from a recent paper \cite{Dyachenko--Nursultanov--Tikhonov--Weisz}. The first one is the Hardy--Littlewood inequality.
\begin{thm}[{\cite{Dyachenko--Nursultanov--Tikhonov--Weisz},  Theorem~4}]\label{Hardy--littlewood inequality on product H^p spaces}
	Assume that $0<p\leq 1$ and $f\in H^p_{\rm{prod}}(\R^d)$, then
		\begin{equation*}
				\bigg(\int_{\R^d}\frac{|\widehat{f}(\xi)|^p}{|\xi_1\xi_2\cdots\xi_d|^{2-p}}\,\mathrm{d}\xi\bigg)^{\frac{1}{p}}
			\lesssim_{d,p}
			\|f\|_{H^p_{\mathrm{prod}}(\R^d)}.
		\end{equation*}
\end{thm}
We give a very sketchy review of the proof of this theorem in \cite{Dyachenko--Nursultanov--Tikhonov--Weisz}.
As a first step, the authors introduced an operator 
\begin{equation}\label{The operator T in Dyachenko et al.'s proof of the Hardy--Littlewood inequality}
	T(f)(\xi)
	\coloneqq
	\Big(\prod_{i=1}^d\xi_i\Big)\widehat{f}(\xi),
\end{equation}
 and then the \textit{Fefferman's boundedness criterion} was applied. For the reader's convenience, we briefly state this criterion. Below, a rectangular $H^p_{\rm{prod}}$-atom is to replace the ball $B$ in the definition of $H^p$-atom with a rectangle $R$, whose edges are parallel to the coordinate axes.
	\begin{namedtheorem}[Fefferman's boundedness criterion]
Assume that $T$ is $L^2(\R^2)$ bounded. Suppose that there exists $\delta>0$, so that for all rectangular $H^1_{\rm{prod}}$-atoms $a$ supported in $R$,
\begin{equation}\label{Fefferman's criterion: decay condition}
	\int_{(\gamma R)^c}|T(a)|\,\mathrm{d}x\lesssim \gamma^{-\delta},\text{ for all }\gamma\geq2,
\end{equation}
where $\gamma R$ is the $\gamma$-fold concentric dilation of $R$, then $T$ is $H^1_{\rm{prod}}\mapsto L^1$ bounded.
	\end{namedtheorem}
This criterion is due to Fefferman (see \cite{Fefferman_PNAS_Papar}, page 840). A direct generalization to $H^1_{\rm{prod}}(\R^3)$ is not true, as there exists a counter-example due to Journ\'e (Section 2 in \cite{Journe_Multiparameter_case}). To apply this criterion to $H^p_{\rm{prod}}(\R^d)$ for $d\geq3$ and $0<p\leq1$, the authors in \cite{Dyachenko--Nursultanov--Tikhonov--Weisz} used a variant form, see Lemma 6 there, or \cite{Weisz_book_Hardy_Spaces}, Chapter~3, Theorem~3.6.12.
The proof of Theorem~3.6.12 in \cite{Weisz_book_Hardy_Spaces}, along with some newly introduced definitions (like the \textit{simple $p$-atoms}), requires some preparatory work. 

Moreover, \textit{Fefferman's criterion} is a result for Lebesgue measure. Due to the nature of the problem here, in Lemma 6 of \cite{Dyachenko--Nursultanov--Tikhonov--Weisz}, the authors apply a (further) variant: replacing the Lebesgue measure in the left hand side of \eqref{Fefferman's criterion: decay condition} by $\mathrm{d}\eta\coloneqq|x_1x_2\cdots x_d|^{-2}\mathrm{d}x$. Therefore, an additional verification of a majorant property (see inequality (38) in \cite{Dyachenko--Nursultanov--Tikhonov--Weisz}) of $\mathrm{d}\eta$ is required.

Two other main results in \cite{Dyachenko--Nursultanov--Tikhonov--Weisz} concerns the Hardy--Ces\`aro operator.
\begin{thm}	\label{Dyachenko et al.'s result on Hardy's inequality of the Hardy--Cesaro operator}
	(Theorem 5, \cite{Dyachenko--Nursultanov--Tikhonov--Weisz})
	Let $\mathcal{H}$ be the Hardy--Ces\`aro operator. Suppose that $f\in H_{\mathrm{prod}}(\R^d)\cap L^2(\R^d)$, then
	\begin{equation*}
		\int_{\R^d}\frac{|\mathcal{H}(\widehat{f})(x)|^p}{|x_1x_2\cdots x_d|^{2-p}}\,\mathrm{d}x
		\lesssim
		\|f\|^p_{H^p_{\mathrm{prod}}}.
	\end{equation*}
\end{thm}
To prove this theorem, the authors in \cite{Dyachenko--Nursultanov--Tikhonov--Weisz} define an operator 
$T_{\mathcal{H}}(f)(\xi)
	\coloneqq
	\Big(\prod_{i=1}^d\xi_i\Big)\mathcal{H}(\widehat{f})(\xi)$, and then proceed as in the proof of Theorem~\ref{Hardy--littlewood inequality on product H^p spaces}.

\begin{thm}\label{Dyachenko et al.'s result on the H^p to L^p boundedness of the Hardy--Cesaro operator}
	(Theorem 6, \cite{Dyachenko--Nursultanov--Tikhonov--Weisz})
	Suppose that $f\in H^p_{\mathrm{prod}}(\R^d)\cap L^2(\R^d)$, then
	\begin{equation*}
		\|\mathcal{H}f\|_p
		\lesssim
		\|f\|_{H^p_{\mathrm{prod}}}.
	\end{equation*}
\end{thm}
\begin{rmk}\label{Failure of H^p to H^p boundeness of the Hardy--Cesaro operator}
	The function $\chi_{(0,1]}-\chi_{(1,2]}$ indicates that $\mathcal{H}$ is not $H^p_{\mathrm{prod}}\mapsto H^p_{\mathrm{prod}}$ bounded.
\end{rmk}

As a byproduct of Proposition~\ref{H^p to L^p Boundedness of the iterated Hilbert transform}, we give shorter proofs of several main results in \cite{Dyachenko--Nursultanov--Tikhonov--Weisz}, which does not rely on Calder\'on--Zygmund theory on product domain. The following claims will be used.
\begin{prop}\label{Han--Li--Lu--Wang's L^p and H^p marjoration}
	For $p\in(0,1]$, if $f\in L^2(\R^d)\cap H^p(\R^d)$, then $\|f\|_{L^p}$ is well defined and $\|f\|_{L^p}\lesssim\|f\|_{H^p}$. The same conclusion holds when $H^p$ is replaced by $H^p_{\mathrm{prod}}$.
\end{prop}
Since $f\in L^2$, the Poisson maximal function $\sup_{\delta>0}|f\ast P_\delta|\in L^2$. We then subtract an a.e. convergent sequence and use Fatou's lemma to obtain $L^p$-integrability of $f$. For details, see \cite{Han--Li--Lu--Wang}, proofs of Theorem ~1.1 and Theorem~1.2.

\begin{prop}[\cite{Uchiyama}, Corollary~17.1]\label{Uchiyama's characterization on H^p space}
	 For $p\in(0,1]$, if $f\in L^2(\R)\cap H^p(\R)$, then
	\begin{equation*}
		\|f\|_{H^p(\R)}\approx\|f\|_{L^p(\R)}+\|Hf\|_{L^p(\R)},
	\end{equation*}
where $H$ is the Hilbert transform.
\end{prop}
The monogrpah \cite{Uchiyama} states that $\|f\|_{H^p(\R)}\lesssim\|f\|_{L^p(\R)}+\|Hf\|_{L^p(\R)}$,
	but by Proposition~\ref{Han--Li--Lu--Wang's L^p and H^p marjoration}, we know that
$\|f\|_{L^p(\R)}\lesssim	\|f\|_{H^p(\R)},$ and $	\|Hf\|_{L^p(\R)}
\lesssim
\|f\|_{H^p(\R)}$.
The $L^2(\R)$-integrability of $f$ ensures that (by Proposition~\ref{Han--Li--Lu--Wang's L^p and H^p marjoration}) the quantity $\|f\|_{L^p(\R)}$ is well-defined.

We prove a slightly stronger form of Theorem~\ref{Hardy--littlewood inequality on product H^p spaces}.
	\begin{thm}\label{Theorem: Hardy's inequality on product H^p spaces}
		For $p\in(0,1]$ and $f\in L^2(\R^3)\cap H^p_{\rm{prod}}(\R^3)$, the following inequality is true:
			\begin{equation*}
			\int_{\R^3}\frac{|\widehat{f}(t)|^p}{|t_1t_2t_3|^{2-p}}\,\mathrm{d}t
			\lesssim
			\|f\|_p+
			\|H_1H_2H_3f\|^p_p+
			\sum_{1\leq i\leq3}\|H_if\|^p_p+
			\sum_{1\leq i < j\leq3}\|H_iH_jf\|^p_p.
		\end{equation*}
	\end{thm}
\begin{rmk}
We use $d=3$ here just because the $3$-dimensional case is typical enough, and the notations are still not too complicated. When $d=2$, it becomes
\begin{equation*}
	\int_{\R^2}\frac{|\widehat{f}(t_1,t_2)|^p}{|t_1|^{2-p}|t_2|^{2-p}}\,\mathrm{d}t
	\lesssim
	\|f\|^p_{p}
	+
	\|H_1f\|^p_{p}
	+
	\|H_2f\|^p_{p}
	+
	\|H_1H_2f\|^p_{p}.
\end{equation*}
The reader can easily write down the statement for $d\in\N$. Also, this theorem is not new when $p=1$, which is due to  Angeloni, Liflyand and Vinti, see \cite{Angeloni--Liflyand--Vinti}, Proposition 1.
\end{rmk}
	\begin{proof}
	By Proposition~\ref{Han--Li--Lu--Wang's L^p and H^p marjoration}, for $f\in H_{\mathrm{prod}}^p\cap L^2$ we have $\|f\|_{p}^p\lesssim\|f\|^p_{H_{\mathrm{prod}}^p}$. Interpolation shows that $f\in L^1$.  Now the Fubini--Tonelli's theorem allows us to write the integral as 
		\begin{equation*}
			\int_{\R^2}\frac{1}{|t_1t_2|^{2-p}}\bigg(\int_\R\frac{|\mathcal{F}_3\mathcal{F}_2\mathcal{F}_1f|^p}{|t_3|^{2-p}}\,\mathrm{d}t_3\bigg)\mathrm{d}t_2\mathrm{d}t_1.
		\end{equation*}
		The one-dimensional Hardy--Littlewood inequality, and Proposition~\ref{Uchiyama's characterization on H^p space} imply that, for a.e. $t_1,t_2\in\R$,
		\begin{equation*}
			\begin{aligned}
				\int_\R\frac{|\mathcal{F}_3\mathcal{F}_2\mathcal{F}_1f|^p}{|t_3|^{2-p}}\,\mathrm{d}t_3
				\lesssim
				\|\mathcal{F}_2\mathcal{F}_1f(t_1,t_2,z)\|^p_{L^p(\R,\mathrm{d}z)}
				+
				\|H_3\mathcal{F}_2\mathcal{F}_1f(t_1,t_2,z)\|^p_{L^p(\R,\mathrm{d}z)},
			\end{aligned}
		\end{equation*}
In order to use Proposition~\ref{Uchiyama's characterization on H^p space}, we employed the following corollary of Fubini's theorem and Plancherel's identity: For a.e. $t_1,t_2\in\R$, the following function
		\begin{equation}\label{The z variable almost everywhere L^2 boundedness of the function}
			z\mapsto \mathcal{F}_2\mathcal{F}_1f(t_1,t_2,z)
			=
			\int_{\R^2}f(x,y,z)\mathrm{e}^{-2\pi\mathrm{i}(xt_1+yt_2)}\,\mathrm{d}x\mathrm{d}y
		\end{equation}
is $L^2(\R,\mathrm{d}z)$-integrable.

Repeating the above argument to the variable $t_1$ and $t_2$ shows that
	\begin{equation*}
		\begin{aligned}
		\int_{\R^2}	\frac{\|\mathcal{F}_2\mathcal{F}_1f(t_1,t_2,z)\|^p_{L^p(\R,\mathrm{d}z)}}{|t_1t_2|^{2-p}}\,\mathrm{d}t_1\mathrm{d}t_2
		&\lesssim
		\|f\|^p_{L^p(\R^3)}
		+
		\|H_1f\|^p_{L^p(\R^3)}\\
		&+
		\|H_2f\|^p_{L^p(\R^3)}
		+
		\|H_1H_2f\|^p_{L^p(\R^3)},\\
		\int_{\R^2}	\frac{\|H_3\mathcal{F}_2\mathcal{F}_1f(t_1,t_2,z)\|^p_{L^p(\R,\mathrm{d}z)}}{|t_1t_2|^{2-p}}\,\mathrm{d}t_1\mathrm{d}t_2
			&\lesssim
			\|H_3f\|^p_{L^p(\R^3)}
			+
			\|H_1H_3f\|^p_{L^p(\R^3)}\\
			&+
			\|H_2H_3f\|^p_{L^p(\R^3)}
			+
			\|H_1H_2H_3f\|^p_{L^p(\R^3)}.
		\end{aligned}
	\end{equation*}
	\end{proof}
Too see why Theorem~\ref{Theorem: Hardy's inequality on product H^p spaces} implies Theorem~\ref{Hardy--littlewood inequality on product H^p spaces}, let us take $d=3$ for example. Assume that $f\in H^p_{\rm{prod}}\cap L^2$, by Proposition~\ref{H^p to L^p Boundedness of the iterated Hilbert transform} and Proposition~\ref{Han--Li--Lu--Wang's L^p and H^p marjoration}, we know that $\|f\|_p\lesssim\|f\|_{H^p_{\rm{prod}}}$, and $\|H_1H_2H_3f\|_p\lesssim\|H_1H_2H_3f\|_{H^p_{\rm{prod}}}\lesssim\|f\|_{H^p_{\rm{prod}}}$. Now it suffices to show that
\begin{equation*}
	\|H_if\|_p
	\lesssim
	\|f\|_{H^p_{\rm{prod}}},
	\text{ and }
	\|H_iH_jf\|_p
	\lesssim
	\|f\|_{H^p_{\rm{prod}}},
	\text{ whenever }i\neq j.
\end{equation*}
We consider the double Hilbert transform $H_1H_2$, the single operator $H_i$ can be proved by the same way. By definition,
	\begin{equation*}
		\begin{aligned}
			\|H_1H_2f\|_p^p
			=
			\int_{\R}\bigg(\int_{\R^2}|H_1H_2f|^p(x,y,z)\,\mathrm{d}x\mathrm{d}y\bigg)\mathrm{d}z.
		\end{aligned}
	\end{equation*}
	Fixing $z$, we view $f$ as a function $(x,y)\mapsto f_z(x,y)\coloneqq f(x,y,z)$. Notice that for almost everywhere $z\in\R$, the function $f_z\in L^2(\R^2)$ by Fubini's theorem. Therefore, by applying the 2-dimensional case of Proposition~\ref{H^p to L^p Boundedness of the iterated Hilbert transform}, we get, for almost everywhere $z\in\R$,
	\begin{equation*}
		\begin{aligned}
			\int_{\R^2}|H_1H_2f|^p(x,y,z)\,\mathrm{d}x\mathrm{d}y
			=
			\|H_1H_2f_z\|^p_{L^p(\R^2)}
			\lesssim
			\|f_z\|^p_{H^p_{\rm{prod}}(\R^2)}.
		\end{aligned}
	\end{equation*}
 In the introduction, the reader has seen the definition of $H^p_{\rm{prod}}$ by Poisson maximal function:
	\begin{equation*}
		\|f_z\|^p_{H^p_{\rm{prod}}(\R^2)}
		=
		\int_{\R^2}\sup_{\delta_1,\delta_2>0}|P_{\delta_1,\delta_2}\ast  f_z|^p\,\mathrm{d}(x,y).
	\end{equation*}
Obviously, the quantity $P_{\delta_1,\delta_2}\ast  f_z$ can be rewritten as
	\begin{equation*}
		\begin{aligned}
\int_{\R^3}P_{\delta_1}(u)P_{\delta_2}(v)P_{0}(w) f(x-u,y-v,z-w)\,\mathrm{d}u\mathrm{d}v\mathrm{d}w.
		\end{aligned}
	\end{equation*}
	Here we used the fact $(\delta_0\ast f)(z)=f(z)$ and $P_0\coloneqq\delta_0$. Taking supremum with respect to $\delta_1,\delta_2>0$, we have, for almost everywhere $z\in\R$, 
	\begin{equation*}
		\begin{aligned}
			\sup_{\delta_1,\delta_2>0}|P_{\delta_1,\delta_2}\ast_{1,2}  f_z|(x,y)
			&=
			\sup_{\delta_1,\delta_2>0}\bigg|\int_{\R^3}P_{\delta_1}(u)P_{\delta_2}(v)P_{0}(w) f(x-u,y-v,z-w)\,\mathrm{d}u\mathrm{d}v\mathrm{d}w\bigg|\\
			&=
			\sup_{\delta_1,\delta_2>0}|P_{\delta_1,\delta_2,0}\ast_{1,2,3} f|(x,y,z)\\
			&\leq
			\sup_{\delta_1,\delta_2,\delta_3>0}|P_{\delta_1,\delta_2,\delta_3}\ast_{1,2,3} f|(x,y,z)
		\end{aligned}
	\end{equation*}
	Here, the convolution $\ast_{1,2}$ is taken with respect to $x,y$ (the first and second coordinates), while $\ast_{1,2,3}$ is taken for $x,y,z$. Also, when passing through $\sup_{\delta_1,\delta_2>0}|P_{\delta_1,\delta_2,0}\ast_{1,2,3} f|$ to $\sup_{\delta_1,\delta_2,\delta_3>0}|P_{\delta_1,\delta_2,\delta_3}\ast_{1,2,3} f|$, we used the Lebesgue differential theorem:
	\begin{equation*}
		|f(z)|
		=
		\lim_{\delta\rightarrow0}|P_{\delta}\ast f|(z)
		\leq
		\sup_{\delta>0}|P_{\delta}\ast f|(z),\text{ for a.e. }z\in\R.
	\end{equation*} 
	Combining these estimates gives us
	\begin{equation*}
		\begin{aligned}
			\|H_1H_2f\|_{L^p(\R^3)}^p
			&\lesssim
			\int_{\R}\bigg(\iint_{\R^2}\sup_{\delta_1,\delta_2>0}|P_{\delta_1,\delta_2,0}\ast
				f|^p\,\mathrm{d}x\mathrm{d}y\bigg)\,\mathrm{d}z\\
			&\leq
			\int_{\R^3}	\sup_{\delta_1,\delta_2,\delta_3>0}|P_{\delta_1,\delta_2,\delta_3}\ast f|^p\,\mathrm{d}x\mathrm{d}y\mathrm{d}z\\
			&=
			\|f\|^p_{H^p_{\rm{prod}}(\R^3)}.
		\end{aligned}
	\end{equation*}
The proofs of Theorem~\ref{Dyachenko et al.'s result on Hardy's inequality of the Hardy--Cesaro operator} and Theorem~\ref{Dyachenko et al.'s result on the H^p to L^p boundedness of the Hardy--Cesaro operator} are similar to that of Theorem~\ref{Theorem: Hardy's inequality on product H^p spaces}.
	\begin{proof}[proof of Theorem~\ref{Dyachenko et al.'s result on Hardy's inequality of the Hardy--Cesaro operator}]
	For simplicity, we assume $d=2$. The right hand side is
		\begin{equation*}
			\int_{\R}\frac{1}{|x_1|^2}\int_{\R}\frac{1}{|x_2|^2}\bigg|\int_{0}^{x_1}\int_{0}^{x_2}\widehat{f}(t)\,\mathrm{d}t_1\mathrm{d}t_2\bigg|^p\,\mathrm{d}x_1\mathrm{d}x_2.
		\end{equation*}
		We write the innermost integral $\int_{0}^{x_1}\int_{0}^{x_2}$ as
		\begin{equation*}
			\begin{aligned}
\int_{0}^{x_2}\int_{\R}
				\bigg(\int_{0}^{x_1}
				\int_{\R}f(y_1,y_2)\mathrm{e}^{-2\pi\mathrm{i}y_1t_1}\,\mathrm{d}y_1\mathrm{d}t_1\bigg)\mathrm{e}^{-2\pi\mathrm{i}y_2t_2}\,
				\mathrm{d}y_2\mathrm{d}t_2
				=
				\int_{0}^{x_2}\widehat{f_{x_1}}(t_2)\,\mathrm{d}t_2,
			\end{aligned}
		\end{equation*}
		where $f_{x_1}(y_2)\coloneqq
			\int_{0}^{x_1}
			\int_{\R}f(y_1,y_2)\mathrm{e}^{-2\pi\mathrm{i}y_1t_1}\,\mathrm{d}y_1\mathrm{d}t_1$.
		By one-dimensional result\footnote{According to \cite{Dyachenko--Nursultanov--Tikhonov--Weisz}, for this theorem, even the 1-dimensional case is a new result. However, the  1-dimensional case can be easily proved by atomic decomposition, so we omit it here and admit it as true.}
		\begin{equation*}
			\begin{aligned}
				\int_{\R}\frac{1}{|x_2|^{2-p}}\,\bigg|\frac{1}{x_2}\int_{0}^{x_2}\widehat{f_{x_1}}(t_2)\,\mathrm{d}t_2\bigg|^p\,\mathrm{d}x_2
				\lesssim
				\|f_{x_1}\|^p_{H^p(\R,\mathrm{d}y_2)}.
			\end{aligned}
		\end{equation*}
		It is easy to show that for almost everywhere $x_1\in \R$, the function $y_2\mapsto f_{x_1}(y_2)$ belongs to $ L^2(\R,\mathrm{d}y_2)$, so $\|f_{x_1}\|^p_{H^p(\R,\mathrm{d}y_2)}
			\approx
			\|f_{x_1}\|^p_{L^p(\R,\mathrm{d}y_2)}
			+
			\|H_2f_{x_1}\|^p_{L^p(\R,\mathrm{d}y_2)}$.
		By previous estimate, we have
		\begin{equation*}
			\begin{aligned}
				\int_{\R}\frac{1}{|x_1|^2}\int_{\R}\frac{1}{|x_2|^2}\bigg|\int_{0}^{x_1}\int_{0}^{x_2}\widehat{f}(t)\,\mathrm{d}t_1\mathrm{d}t_2\bigg|^p\,\mathrm{d}x_1\mathrm{d}x_2
				\lesssim
				\int_{\R}\frac{	\|f_{x_1}\|^p_{L^p(\R,\mathrm{d}y_2)}
				+\|H_2f_{x_1}\|^p_{L^p(\R,\mathrm{d}y_2)}}{|x_1|^2}\mathrm{d}x_1.
			\end{aligned}
		\end{equation*}		
		Again, one-dimensional result gives
		\begin{equation*}
			\begin{aligned}
				\int_{\R}\frac{	\|f_{x_1}\|^p_{L^p(\R,\mathrm{d}y_2)}
				}{|x_1|^2}\mathrm{d}x_1
			\lesssim	
				\|f\|_{p}^p
				+
				\|H_1f\|_{p}^p.
			\end{aligned}
		\end{equation*}
The rest of the proof is similar to that of the previous theorem.
	\end{proof}

	\begin{proof}[proof of Theorem~\ref{Dyachenko et al.'s result on the H^p to L^p boundedness of the Hardy--Cesaro operator}]
		Still, we  consider the case $d=2$. We have
		\begin{equation*}
			\begin{aligned}
				\|\mathcal{H}f\|^p_p
				=
				\int_{\R}\frac{1}{|x_1|^p}	\int_{\R}\frac{1}{|x_2|^p}\bigg|\int_{0}^{x_2}f_{x_1}(t_2)\mathrm{d}t_2\bigg|^p\,\mathrm{d}x_2\mathrm{d}x_1,\\
			\end{aligned}
		\end{equation*}
		where $f_{x_1}(t_2)\coloneqq\int_{0}^{x_1}f(t_1,t_2)\,\mathrm{d}t_1$. The one-dimensional inequality\footnote{According to \cite{Dyachenko--Nursultanov--Tikhonov--Weisz}, the 1-dimensional inequality is classical, so we use it freely here.} states that
		\begin{equation*}
			\int_{\R}\frac{1}{|x_2|^p}\bigg|\int_{0}^{x_2}f_{x_1}(t_2)\mathrm{d}t_2\bigg|^p\,\mathrm{d}x_2
			\lesssim
			\|f_{x_1}\|^p_{H^p(\R,\mathrm{d}y_2)}
			\approx
			\|f_{x_1}\|^p_{L^p(\R,\mathrm{d}y_2)}
			+
			\|H_2f_{x_1}\|^p_{L^p(\R,\mathrm{d}y_2)}.
		\end{equation*}
	The rest of the proof is similar to that of the previous theorems.
	\end{proof}
\section{Acknowledgement}
This paper was completed at Nicolaus Copernicus University in Toruń, where the author was a postdoctoral researcher at that time. The author would like to thank Yuri Tomilov for introducing him to the theory of Hardy spaces and Hardy type inequalities.
\printbibliography
\end{document}